\documentclass[12pt]{nyjm} 

\usepackage{amssymb}
\usepackage{hyperref}
\usepackage[capitalize]{cleveref}

\renewcommand{\MR}[1]{\href{http://www.ams.org/mathscinet-getitem?mr=#1}{MR#1}}
\newcommand{\Zbl}[1]{\href{http://zbmath.org/?q=an:#1}{Zbl~#1}}

\title[Horospherical limit points of locally symmetric spaces]
{Horospherical limit points of \\ finite-volume locally symmetric spaces}

\author{Grigori Avramidi}
\address{Department of Mathematics, University of Chicago, Chicago, IL 60637}
\email{gavramid@math.uchicago.edu, http://math.uchicago.edu/$\sim$gavramid/}

\author{Dave Witte Morris}
\address{Department of Mathematics and Computer Science, University of Lethbridge, Lethbridge, Alberta, T1K\,6R4, Canada}
\email{Dave.Morris@uleth.ca, http://people.uleth.ca/$\sim$dave.morris/}

\date{\today}

\DeclareMathOperator{\SL}{SL}
\DeclareMathOperator{\unip}{unip}
\DeclareMathOperator{\tr}{tr}
\DeclareMathOperator{\ad}{ad}

\newcommand{\real}{\mathbb{R}}
\newcommand{\rational}{\mathbb{Q}}
\newcommand{\integer}{\mathbb{Z}}
\newcommand{\closure}{\overline}

\renewcommand{\AA}{\mathbf{A}}
\newcommand{\GG}{\mathbf{G}}
\newcommand{\PP}{\mathbf{P}}

\renewcommand{\SS}{\mathbf{S}} 

\newcommand{\rank}[1]{\mathop{\text{rank}}_{#1}}

\newcommand{\bdry}{\partial}

\numberwithin{equation}{section}

\theoremstyle{plain}
\newtheorem{thm}[equation]{Theorem}
\newtheorem{prop}[equation]{Proposition}
\newtheorem{cor}[equation]{Corollary}
\newtheorem{lem}[equation]{Lemma}

\crefname{cor}{Corollary}{Corollaries}
\Crefname{cor}{Corollary}{Corollaries}
\crefname{lem}{Lemma}{Lemmas}
\Crefname{lem}{Lemma}{Lemmas}

\theoremstyle{definition}
\newtheorem{notation}[equation]{Notation}
\newtheorem{defn}[equation]{Definition}
\newtheorem*{ack}{Acknowledgments}

\theoremstyle{definition}
\newtheorem{rem}[equation]{Remark}
\newtheorem{rems}[equation]{Remarks}

\crefformat{rems}{Remark~#2#1#3}

 \newcounter{case}
 
 \renewcommand{\thecase}{\arabic{case}}
 
 \renewcommand{\Cref}{\cref}
 \newcommand{\pref}[1]{(\ref{#1})}
 \newcommand{\fullcref}[2]{\cref{#1}\pref{#1-#2}}
 \newcommand{\fullCref}[2]{\Cref{#1}\pref{#1-#2}}

\makeatletter
\newcommand{\noprelistbreak}{\smallskip\@nobreaktrue\nopagebreak} 
\makeatother

\makeatletter
\def\cref@thmoptarg[#1]#2#3#4{%
  \ifhmode\unskip\unskip\par\fi%
  \normalfont%
  \trivlist%
  \let\thmheadnl\relax%
  \let\thm@swap\@gobble%
  \thm@notefont{\fontseries\mddefault\upshape}%
  \thm@headpunct{.}
  \thm@headsep 5\p@ plus\p@ minus\p@\relax%
  \thm@space@setup%
  #2
  \@topsep \thm@preskip               
  \@topsepadd \thm@postskip           
  \def\@tempa{#3}\ifx\@empty\@tempa%
    \def\@tempa{\@oparg{\@begintheorem{#4}{}}[]}%
  \else%
    \refstepcounter[#1]{#3}
    \def\dth@counter{} 
    \def\@tempa{\@oparg{\@begintheorem{#4}{\csname the#3\endcsname}}[]}%
  \fi%
  \@tempa}
\makeatother

\begin{document}

\begin{abstract}
Suppose $X/\Gamma$ is an arithmetic locally symmetric space  of noncompact type (with the natural metric induced by the Killing form of the isometry group of~$X$), and let $\xi$ be a point on the visual boundary of~$X$. T.\,Hattori showed that if each horoball based at~$\xi$ intersects every $\Gamma$-orbit in~$X$, then $\xi$ is not on the boundary of any $\rational$-split flat in~$X$. We prove the converse. (This was conjectured by W.\,H.\,Rehn in some special cases.)
Furthermore, we prove an analogous result when $\Gamma$ is a nonarithmetic lattice.
\end{abstract}

\maketitle

\setcounter{tocdepth}{1} 
\tableofcontents

\section{Introduction}

\begin{defn}[{}{\cite[Defn.~B]{Hattori-GeomLimSets}}] \label{HoroLimDefn}
Let $X/\Gamma$ be a locally symmetric space of noncompact type (with universal cover~$X$), and let $x \in X$. A point~$\xi$ on the visual boundary of~$X$ is a \emph{horospherical limit point} for~$\Gamma$ if every horoball based at~$\xi$ intersects the orbit $x \cdot \Gamma$.
(See \cref{CGammaHoroLimDefn} for an alternate characterization which makes it clear that this notion is independent of the choice of the basepoint~$x$.)
\end{defn}

Our main theorem characterizes the horospherical limit points for any finite-volume locally symmetric space $X/\Gamma$ of noncompact type.  The result is slightly easier to state if we assume that the lattice~$\Gamma$ is arithmetic. (See \cref{NonarithSect} for the general case.) 

\begin{defn}
Let $G$ be the real points of a connected, semisimple algebraic group over~$\rational$, and let $X = K \backslash G$ be the corresponding symmetric space of noncompact type.
\noprelistbreak
	\begin{enumerate}
	\item It is well known that if $T$ is any $\real$-split torus in~$G$, then there exists $x \in X$, such that $x T$ is a flat in~$X$ (cf.\ \cite[Prop.~6.1, pp.~245]{Helgason-DiffGeom}). 
	We say the flat $x T$ is \emph{$\rational$-split} if the torus~$T$ is (defined over~$\rational$ and) $\rational$-split.
	\item The Killing form on~$G$ induces a metric on $K \backslash G$ that gives it the structure of a symmetric space \cite[Prop.~3.6]{Helgason-LieGrpSymSp}. 
	We call this the \emph{Killing-form metric}. See \cref{NotKilling} for a formulation of our results that applies to the other symmetric metrics on $K \backslash G$.
	\end{enumerate}
\end{defn}

The direction ($\Rightarrow$) of the following result has already been proved by T.\,Hattori \cite[Thm.~A or Prop.~4.4]{Hattori-GeomLimSets}, but we provide a proof of both directions because our methods are quite different.

\begin{thm} \label{HoroLimIffgRat}
Let $X / \Gamma$ be an arithmetic locally symmetric space of noncompact type with the Killing-form metric. A point $\xi \in \bdry X$ is a horospherical limit point for\/~$\Gamma$ if and only if $\xi$~is not on the boundary of any\/ $\rational$-split flat.
\end{thm}

Since $\GG(\rational)$ acts transitively on the set of maximal $\rational$-split tori, we have the following reformulation:

\begin{cor} \label{RatBldg}
Let 
\noprelistbreak
	\begin{itemize}
	\item $G$ be the real points of a connected, semisimple algebraic group over\/~$\rational$, 
	\item $X = K \backslash G$ be the corresponding symmetric space of noncompact type with the Killing-form metric, 
	and 
	\item $B$ be the boundary of some maximal\/ $\rational$-split flat in~$X$. 
	\end{itemize}
Then the set of horospherical limit points for\/~$G_{\integer}$ is the complement of\/ $\bigcup_{g \in G_{\rational}} Bg$.
\end{cor}

For the special case of $\rational$-split groups, we can state this another way:

\begin{cor} \label{QsplitHoroLimits}
Let $\GG$ be a connected, $\rational$-split, semisimple algebraic group over\/~$\rational$.
A point~$\xi$ on the visual boundary of the corresponding symmetric space $X = K \backslash \GG(\real)$ is \textbf{not} a horospherical limit point for~$\GG(\integer)$ if and only if $\xi$ is fixed by some parabolic $\rational$-subgroup of~$\GG$.
\end{cor}

\begin{rems} \label{GenRems} \ 
\noprelistbreak
	\begin{enumerate} \itemsep=\smallskipamount

	\item The set $\bigcup_{g \in G_{\rational}} Bg$ in the statement of \cref{RatBldg} is known as the ``rational Tits building'' of~$\GG$ \cite[p.~324]{Leuzinger-Tits}. Thus, the result states that the set of horospherical limit points of $\GG(\integer)$ is equal to the complement of the rational Tits building of~$\GG$.
	This was conjectured by W.\,H.\,Rehn \cite{RehnThesis} (in somewhat less generality), but the inclusion ($\supset$) has remained open even for the case where $\GG(\integer) = \SL_n(\integer)$ with $n \ge 3$. 
	
	\item A geodesic ray $\gamma^+$ is \emph{divergent} if the function $\gamma^+ \colon \real^+ \to X/\Gamma$ is a proper map. It is easy to see that if the endpoint of~$\gamma^+$ is not a horospherical limit point, then $\gamma^+$ must be divergent. The converse is not true, because S.\,G.\,Dani \cite{Dani-Divergent} has shown that if $\rank{\real} X \ge 2$, then there are many geodesic rays that diverge for ``non-obvious'' reasons, and \cref{RatBldg} shows that the endpoints of such rays are horospherical limit points.

	\item In \cref{QsplitHoroLimits}, the assumption that $\GG$ is $\rational$-split can be weakened to the assumption that $\rank{\rational} \GG = \rank{\real} \GG$.
	We also note that this \lcnamecref{QsplitHoroLimits} does not assume $X$~has the Killing-form metric --- it is valid for every symmetric metric on $K \backslash \GG(\real)$ (if $\rank{\rational} \GG = \rank{\real} \GG$).

	\item Given a locally symmetric space $X/\Gamma$ and a finitely generated $\Gamma$-module $A$, a corresponding set $\Sigma_\Gamma( X ; A )$ of \emph{horospherical limit points} has been defined by R.\,Bieri and R.\,Geoghegan \cite{BieriGeoghegan}. It reduces to \cref{HoroLimDefn} when $A = \integer$ is the trivial $\Gamma$-module, but it would be interesting to extend \cref{HoroLimIffgRat} by calculating $\Sigma_\Gamma( X ; A )$ for other $\Gamma$-modules.
	
	\end{enumerate}
\end{rems}

The proof of \cref{HoroLimIffgRat} is short (about a page for each direction), but relies on definitions and other background material from the theory of algebraic groups, Lie groups, and unipotent dynamics.
These preliminaries are presented in \cref{PrelimSect}.
\Cref{OnFlatNotHoroSect} proves that the boundary points of a $\rational$-split flat are not horospherical. The other direction of \cref{HoroLimIffgRat} is proved in \cref{MainSect}. (See \cref{HoroLimIff} for a summary that provides several alternative formulations of \cref{HoroLimIffgRat}.) The \hyperref[NonarithSect]{final section} presents a generalization of \cref{HoroLimIffgRat} that  allows $\Gamma$ to be non-arithmetic.

See \cite{MorrisWortman} for a generalization of \cref{HoroLimIffgRat} that allows $\Gamma$ to be an $S$-arithmetic group.

\begin{ack}
We thank Ross Geoghegan for explaining the conjecture of Rehn that motivated this line of research, and we thank the Park City Mathematics Institute for bringing the two of us together and providing an opportunity to start work on this problem. 
We also thank Tam Nguyen Phan and Kevin Wortman for helpful conversations about the structure of horospheres in symmetric spaces of higher rank. In addition, we thank the latter for calling  \cite{Hattori-GeomLimSets} to our attention, and pointing out that it proves one direction of \cref{HoroLimIffgRat}. 

\end{ack}

\section{Preliminaries} \label{PrelimSect}
\begin{notation}
For any Lie group~$H$, we let $H^{\circ}$ be the identity component of~$H$.
\end{notation}

\begin{notation}
$H^g = g^{-1} H g$.
\end{notation}

\subsection{Horospherical limit points}

We record a few well-known, elementary observations.

\begin{lem} \label{CGammaHoroLimDefn}
$\xi$ is a horospherical limit point for\/~$\Gamma$ iff there is a compact subset~$C$ of~$X$, such that $C \cdot \Gamma$ intersects every horoball based at~$\xi$.
\end{lem}

\begin{proof}
($\Rightarrow$) Let $C = \{x\}$, where $x$ is the basepoint chosen in \cref{HoroLimDefn}.

($\Leftarrow$) Choose $R > 0$, such that $d(x,c) < R$ for all $c \in C$. Any horoball $\mathcal{B}_0$ based at~$\xi$ contains a smaller horoball~$\mathcal{B}_R$, such that the distance from~$\mathcal{B}_R$ to the complement of~$\mathcal{B}$ is greater than~$R$. By assumption, there exist $c \in C$ and $\gamma \in \Gamma$, such that $c \gamma \in \mathcal{B}_R$. Since $d(x\gamma, c\gamma) = d(x,c) < R$, this implies $x \gamma \in \mathcal{B}_0$.
\end{proof}

\Cref{CGammaHoroLimDefn} implies that the set of horospherical limit points is independent of the choice of the basepoint $x \in X$, and also does not change if we replace $\Gamma$ by any finite-index subgroup. Therefore, we have the following consequence:

\begin{cor} \label{InvtUnderCommGamma}
The set of horospherical limit points for\/~$\Gamma$ is invariant under the action of the commensurator group\/ $\mathop{\mathrm{Comm}}_G(\Gamma)$ on $\bdry X$. In particular, if\/ $\Gamma = G_{\integer}$ \textup(and $G$ is defined over\/~$\rational$\textup), then the set of horospherical limit points for~$G_{\integer}$ is invariant under the action of $G_{\rational}$ on $\bdry X$. 
\end{cor}

\begin{lem} \label{HoroSubgrp}
Let 
\noprelistbreak
	\begin{itemize}
	\item $A$ be a maximal $\real$-split torus of~$G$,
	\item $x \in X = K \backslash G$, such that $x A$ is a flat in $X$,
	\item $\{a^t\}$ be a nontrivial one-parameter subgroup of~$A$,
	\item $\xi \in \bdry X$ be the endpoint of the ray $\{x a^t\}_{t=0}^\infty$,
	\item $A_\perp$ be the codimension-one subgroup of~$A$ that is orthogonal to $\{a^t\}$ \textup(with respect to the Killing form\/\textup),
	\item $A^+$ be a Weyl chamber of~$A$ that contains the ray $\{a^t\}_{t=0}^\infty$,
	and
	\item $N$ be the maximal unipotent subgroup of~$G$, such that $a^t ua^{-t} \to e$ as $t \to +\infty$ for all $u \in N$ and all $a$ in the interior of~$A^+$.
	\end{itemize}
Then 
\noprelistbreak
	\begin{enumerate}
	\item \label{HoroSubgrp-subgrp}
	$x a^t A_\perp N$ is a horosphere based at~$\xi$, for each $t \in \real$.
	\item \label{HoroSubgrp-limit}
	$\xi$ is not a horospherical limit point for\/~$\Gamma$ iff 
		$$\lim_{t \to \infty} \, \sup_{g \in a^t A_\perp N} \, \inf_{\gamma \in \Gamma \smallsetminus \{e\}} \| g \gamma g^{-1} - e \| = 0 ,$$
	where $e$ is the identity element of~$G$.
	\end{enumerate}
\end{lem}

\begin{proof}
\pref{HoroSubgrp-subgrp} 
Let $P = C_G \bigl( \{a^t\} \bigr) \, N$. For each $g \in P$, the geodesic ray $\{xa^t g\}_{t \ge 0}$ is at bounded distance from $\{xa^t\}_{t \ge 0}$ (because $\{\,a^t g a^{-t} \mid t \ge 0 \,\}$ is a bounded set). Therefore, $P$ fixes the point~$\xi$, so it acts (continuously) on the set of horospheres based at~$\xi$. Since these horospheres are parametrized by~$\real$, and every continuous homomorphism $P \to \real$ is trivial on~$N$, we conclude that $N$ fixes every horosphere based at~$\xi$. Therefore, $x a^t A_{\perp} N$ is contained in the horosphere through $x a^t$. Since the Iwasawa decomposition $G = KAN$ tells us that $G$ is the disjoint union of these sets (and every point of~$X$ is on a unique horosphere), the set must be the entire horosphere.

\pref{HoroSubgrp-limit}~From~\pref{HoroSubgrp-subgrp}, we know that each horoball based at~$\xi$ is of the form $\bigcup_{t \ge t_0} x a^t A_\perp N$ (for some~$t_0$). Therefore, the equivalence in~\pref{HoroSubgrp-limit} is a restatement of \cref{CGammaHoroLimDefn} (by using \cite[Thm.~1.12, p.~22]{RaghunathanBook}).
\end{proof}

\begin{lem} \label{InSpanMustLarge}
Suppose
\noprelistbreak
	\begin{enumerate}
	\item $v,v_1,\ldots,v_n \in \real^k$, with $v \neq 0$,
	\item \label{InSpanMustLarge-span}
	$v$ is in the span of $\{v_1,\ldots,v_n\}$,
	\item \label{InSpanMustLarge-acute}
	$\langle v \mid v_i \rangle \ge 0$ for all~$i$,
	\item \label{InSpanMustLarge-obtuse}
	$\langle v_i \mid v_j \rangle \le 0$ for $i \neq j$,
	and
	\item $T \in \real^+$.
	\end{enumerate}
Then, for all sufficiently large $t \in \real^+$ and all $w \perp v$, there is some~$i$, such that $\langle tv + w | v_i \rangle > T$.
\end{lem}

\begin{proof}
This is a standard argument. 

From \pref{InSpanMustLarge-span}, we may write $v = \sum_i c_i v_i$ with $c_i \in \real$. 
Also, by passing to a subset, we may assume $\{v_1,\ldots,v_n\}$ is linearly independent, and that $c_i \neq 0$ for every~$i$. Then, by replacing $\real^k$ with the span of  $\{v_1,\ldots,v_n\}$, we may assume that $\{v_1,\ldots,v_n\}$ is a basis.

Permute the elements of $\{v_1,\ldots,v_n\}$ so that the negative values of~$c_i$ come first. That is, there is some $k$ with $c_i < 0$ for $i \le k$ and $c_i > 0$ for $i > k$. Let $z = \sum_{i\le k} c_i v_i$. Then
	$$ \langle z \mid v \rangle 
	= \sum_{i\le k} c_i \langle v_i \mid v \rangle
	= \sum_{i\le k} \bigl( {<0} \bigr) \bigl({\ge 0} \bigr)
	\le 0 $$
and
	$$ \langle z \mid v \rangle 
	= \Big\langle z \mathrel{\Big|} z + \sum_{j > k} c_j v_j \Bigr\rangle
	= \langle z \mid z \rangle + \sum_{i \le k < j} c_i c_j \langle v_i \mid v_j  \rangle
	= \bigl( {\ge0} \bigr) + \sum_{i \le k < j} \bigl({< 0}\bigr) \bigl( {> 0} \bigr) \bigl( {\le 0}\bigr)
	\ge 0 .$$
So we must have equality throughout, which implies $\langle z \mid z \rangle = 0$. Therefore $z = 0$, so we must have $k = 0$ (since $\{v_i\}_{i=1}^n$ is linearly independent). This means $c_i > 0$ for all~$i$.

We claim there is some $\epsilon > 0$, such that, for every $w \perp v$, there exists~$i$, such that $\langle w \mid v_i \rangle \ge \epsilon \|w\|$. Suppose not. Then there must be some nonzero $w \perp v$, such that $\langle w \mid v_i \rangle \le 0$ for all~$i$. So
	$$ 0 = \langle v \mid w \rangle = \sum_{i} c_i \langle v_i \mid w \rangle 
	= \sum_{i} \bigl( {> 0} \bigr) \bigl( {\le 0} \bigr) \le 0 .$$
Hence, we must have $\langle v_i \mid w \rangle = 0$ for all~$i$. Since $\{v_i\}_{i=1}^n$ is a basis, this implies $w = 0$, which is a contradiction.

Since $\{v_i\}$ is a basis (and $v$ is nonzero), we must have $\langle v \mid v_j \rangle \neq 0$ for some~$j$. Then $\langle tv \mid v_j \rangle$ is large whenever $t$~is large. Thus, if the conclusion of the \lcnamecref{InSpanMustLarge} fails to hold, then $\langle w \mid v_j \rangle$ must be large (and negative), so $\|w\|$ must be large. By making it so large that $\epsilon \|w\| \ge T$, and applying the claim of the preceding paragraph, we have 
$\langle tv + w | v_i \rangle \ge 0 + T = T$, as desired.
\end{proof}

\subsection{Parabolic subgroups}

\begin{prop}[``real Langlands decomposition'' {\cite[p.~81]{WarnerBook}}] 
If $P$ is a parabolic subgroup of a connected, semisimple Lie group~$G$ with finite center, then we may write $P = MTU$, where
\noprelistbreak
	\begin{itemize}
	\item $T$ is an $\real$-split torus,
	\item $M$ is a connected, reductive subgroup that centralizes~$T$ and has compact center,
	and
	\item $U$ is the unipotent radical of~$P$.
	\end{itemize}
\end{prop}

\begin{lem} \label{SplitPerpNoChar}
Let $Q$ be a field of characteristic~$0$.
If $H$ is a reductive $Q$-subgroup of an algebraic $Q$-group~$G$, and $H$ has no nontrivial $Q$-characters, then $H$ is orthogonal to every $Q$-split torus~$T$ that centralizes it.
\end{lem}

\begin{proof}
Let $\mathfrak{g}$, $\mathfrak{h}$, and~$\mathfrak{t}$  be the Lie algebra of $G$, $H$, and~$T$, respectively. Consider any minimal $(\mathrm{Ad}_G H)$-invariant $Q$-subspace~$V$ of~$\mathfrak{g}$. 
Since $H$ has no $Q$-characters, it must act on~$V$ via $\SL(V)$, so $\tr \bigl( (\ad h)|_V \bigr) = 0$ for every $h \in \mathfrak{h}$. On the other hand, since $T$ centralizes~$H$ (and is $Q$-split), Schur's Lemma tells us that any $t \in \mathfrak{t}$ acts by a scalar~$\lambda$ on~$V$. Therefore 
	$$\tr \bigl( (\ad h)(\ad t)|_V \bigr) = \lambda \cdot \tr \bigl( (\ad h)|_V \bigr) = \lambda \cdot 0 = 0 .$$
Since $H$ is reductive, we know that $\mathfrak{g}$ 
is the direct sum of such submodules~$V$, so the trace of $(\ad h)(\ad t)$ is~$0$.
This means $\mathfrak{h} \perp \mathfrak{t}$ (with respect to the Killing form).
\end{proof}

\begin{cor} \label{MPerpA}
If $P = MTU$ is a parabolic subgroup of~$G$, then $T$ is orthogonal to~$M$.
\end{cor}

\begin{lem} \label{ANinP}
Let $G = KAN$ be an Iwasawa decomposition of~$G$.
If $P$ is a parabolic subgroup of~$G$, and $N \subset P$, then $A \subset P$.
\end{lem}

\begin{proof}
Let $Q = N_G(N)$ be the normalizer of~$N$, so $Q$ is a (minimal) parabolic subgroup of~$G$, such that $A \subset Q$ and $\unip Q = N$. Since $N \subset P$ and $\unip P$ is normal in~$P$, we know that $N \cdot \unip P$ is a unipotent subgroup. Since $N$ is a maximal unipotent subgroup of~$G$, this implies $\unip P \subseteq N$. In other words, $\unip P \subseteq \unip Q$. Since $P$ and~$Q$ are parabolic subgroups, this implies $Q \subseteq P$ (cf.\ \cite[Prop.~5.3]{TomanovWeiss-ToriAct}). 
So $A \subset Q \subseteq P$.
\end{proof}

\begin{lem} \label{AFixesXi}
Let $A$~be a maximal $\real$-split torus of~$G$, $\xi$ be a point on the visual boundary of $K \backslash G$, and $x \in K \backslash G$. If $A$ fixes~$\xi$, and $xA$~is a\/ \textup(maximal\textup) flat in $K \backslash G$, then $\xi$~is on the boundary of~$xA$.
\end{lem}

\begin{proof}
Let $P = \{\, g \in G \mid \xi \, g = \xi \,\}$, and choose a maximal flat $x_1 A_1$, such that $\xi$~is on the boundary of~$x_1 A_1$. 
Since $A_1$ is abelian, it is clear that $A_1$ fixes~$\xi$, so $A_1 \subseteq P$. Also, since $P$ is a parabolic subgroup (cf.\ the start of the proof of \fullcref{HoroSubgrp}{subgrp}), it is Zariski closed, so any two maximal $\real$-split tori in~$P$ are conjugate. Hence, there is some $g \in P$, such that $A_1^g = A$. Then $\xi = \xi g$ is on the boundary of the flat $x_1 A_1 g = x_1 g A_1^g = x_1 g A$. The uniqueness of the flat fixed by~$A$ implies this flat is $xA$.
\end{proof}

\subsection{Unipotent dynamics}

\begin{thm}[Dani {\cite[Thm.~A and Prop.~1.1(ii)]{Dani-OrbitHoro}}] \label{ClosureIsHomog}
If 
\noprelistbreak
	\begin{itemize}
	\item $N$ is a maximal unipotent subgroup of a connected, semisimple Lie group~$G$, 
	and
	\item $\Gamma$ is a lattice in~$G$,
	\end{itemize}
then there is a closed, connected subgroup~$H$ of~$G$, such that
\noprelistbreak
	\begin{enumerate}
	\item $\closure{N \Gamma} = H \Gamma$,
	\item $H \cap \Gamma$ is a lattice in~$H$,
	\item $N  \subseteq H$,
	and
	\item $N$ acts ergodically on $H \Gamma$, with respect to the $H$-invariant probability measure.
	\end{enumerate}
\end{thm}

We can describe the subgroup~$H$ quite explicitly if the lattice~$\Gamma$ is arithmetic:

\begin{cor}[cf.\ {\cite[Prop~6.1]{Dani-MinHoroFlows}}] \label{ClosureIsPorbit}
Suppose 
\noprelistbreak
	\begin{itemize}
	\item $G = \GG_{\real}^\circ$, where $\GG$ is a connected, semisimple algebraic group over~$\rational$,
	\item $\Gamma$ is a subgroup of finite index in~$\GG_{\integer}$,
	and
	\item $N$ is a maximal unipotent subgroup of~$G$.
	\end{itemize}
Then there is a parabolic $\rational$-subgroup $P$ of~$G$, with real Langlands decomposition $P = MTU$, and a connected, closed, normal subgroup $M^*$ of~$M$, such that 
\noprelistbreak
	\begin{itemize}
	\item $\closure{N \Gamma} = M^* U \Gamma$, 
	and
	\item $N\subseteq M^* U$.
	\end{itemize}
\end{cor}
\begin{rem}
Since $M^* U$ contains the maximal unipotent subgroup~$N$, we know that $M^*$ contains all of the noncompact, simple factors of~$M$. However, it may be missing some of the compact factors.
\end{rem}

\begin{rem}
\Cref{ClosureIsHomog} has been vastly generalized by M.\,Ratner \cite[Thm.~A and Cor.~A]{Ratner-Equidist}.
\end{rem}

\section{Boundary points of a \texorpdfstring{$\rational$}{Q}-split flat are not horospherical} \label{OnFlatNotHoroSect}

\begin{prop}[Hattori {\cite[Thm.~A or Prop.~4.4]{Hattori-GeomLimSets}}] \label{OnFlatNotHoro}
Let
\noprelistbreak
	\begin{itemize}
	\item $G = \GG(\real)^\circ$, where\/ $\GG$ is a connected, semisimple\/ $\rational$-group,
	\item $X = K \backslash G$ be the corresponding symmetric space of noncompact type, with the Killing-form metric,
	\item $S = \SS(\real)^\circ$, where\/ $\SS$ is a maximal\/ $\rational$-split torus of\/~$\GG$,
	\item $x \in X$, such that $x S$ is a\/  \textup($\rational$-split\/\textup) flat in~$X$,
	and
	\item $\{a^t\}$ be a one-parameter subgroup of~$S$.
	\end{itemize}
Then the endpoint of the geodesic ray\/ $\{x a^t\}_{t=0}^\infty$ is not a horospherical limit point for $\GG(\integer)$.
\end{prop}

\begin{proof}
Let:
\noprelistbreak
	\begin{itemize}
	\item $\Phi$ be the system of roots of~$\GG$ with respect to~$\SS$,
	\item $\Delta$ be a base of~$\Phi$, such that $\alpha( a^t) \ge 0$ for all $\alpha \in \Delta$ and all $t > 0$,
	\item $\widehat A$ be a $\rational$-torus in~$G$ that contains some maximal $\real$-split torus~$A$, and also contains~$S$ (such a torus can be constructed by applying \cite[Cor.~3 of \S7.1, p.~405]{PlatonovRapinchukBook} to $C_G(S)$),
	and
	\item $A_\perp$ be the orthogonal complement of~$\{a^t\}$ in~$A$.
	\end{itemize}
For each $\alpha \in \Delta$, let:
\noprelistbreak
	\begin{itemize}
	\item $\alpha^A \in S$, such that $\langle a \mid \alpha^A \rangle = \alpha(a)$ for all $a \in S$,
	and
	\item $P_\alpha = S_\alpha M_\alpha N_\alpha$ be the parabolic $\rational$-subgroup of~$G$ corresponding to~$\alpha$, where
\noprelistbreak
		\begin{itemize}
		\item $S_\alpha$ is the one-dimensional subtorus of~$S$ on which all roots in $\Delta \smallsetminus \{\alpha\}$ are trivial,
		\item $M_\alpha$ is reductive with $\rational$-anisotropic center,
		and
		\item the unipotent radical $N_\alpha$ is generated by the roots in~$\Phi^+$ that are \emph{not} trivial on~$S_\alpha$.
		\end{itemize}
	\end{itemize}
Let $N$ be a maximal unipotent subgroup of~$G$ that is normalized by~$A$ and is contained in the minimal parabolic $\rational$-subgroup $\bigcap_{\alpha \in \Delta} P_\alpha$. (In other words, let $N$ be the unipotent radical of a minimal parabolic $\real$-subgroup of~$G$ that contains~$A$ and is contained in $\bigcap_{\alpha \in \Delta} P_\alpha$.)

Note that:
\noprelistbreak 
	\begin{itemize}
	\item Since $\Delta$ is a basis for the dual of~$S$ (viewed as a vector space), we know that $\{\alpha^A\}_{\alpha \in \Delta}$ spans~$S$. Hence, $\{a^t\}$ is contained in the span of~$\{\alpha^A\}_{\alpha \in \Delta}$.
	\item For $\alpha \in \Delta$ and $t \in \real^+$, we have $\langle a^t \mid \alpha^A \rangle = \alpha(a^t) \ge 0$. 
	\item For $\alpha , \beta \in \Delta$ with $\alpha \neq \beta$, it is a basic property of root systems that $\langle \alpha \mid \beta \rangle \le 0$. Therefore $\langle \alpha^A \mid \beta^A \rangle \le 0$. 
	\end{itemize}
So \cref{InSpanMustLarge} tells us that if $t \in \real^+$ is sufficiently large, then, for all $b \in A_\perp$, there exists $\alpha \in \Delta$, such that $\langle a^t b \mid \alpha^A \rangle$ is large. 

Note that $\alpha$ extends uniquely to a $\rational$-character~$\widehat\alpha$ of~$\widehat A$. Namely, $\widehat\alpha$ must be trivial on the $\rational$-anisotropic part of~$\widehat A$, which is complementary to~$S$. Then, since \cref{SplitPerpNoChar} tells us that the anisotropic part is orthogonal to~$S$, we have $\langle a \mid \alpha^A \rangle = \widehat\alpha(a)$ for all $a \in \widehat A$ (not only for $a \in S$). Hence, the conclusion of the preceding paragraph tells us that $\widehat\alpha(a^t b)$ is large.

Since conjugation by the inverse of $a^t  b$ contracts the Haar measure on $N_\alpha$ by a factor of $\alpha( a^t  b )^k$ for some $k \in \integer^+$, and the action of~$N$ on $N_\alpha$ is volume-preserving, this implies that, for any $g \in a^t b N$, conjugation by the inverse of~$g$ contracts the Haar measure on~$N_\alpha$ by a large factor. Since $(N_\alpha)_{\integer}$ is a cocompact lattice in $N_\alpha$ \cite[Thm.~2.12]{RaghunathanBook}, this implies there is some nontrivial $h \in (N_\alpha)_{\integer}$, such that 
	$\| g h g^{-1} - e \|$ is small.
Therefore, \fullcref{HoroSubgrp}{limit} implies that $\xi$ is not a horospherical limit point for $\GG(\integer)$.
\end{proof}

\section[Non-horospherical limit points are on a \texorpdfstring{$\rational$}{Q}-split flat]{Non-horospherical limit points are on the boundary of a \texorpdfstring{$\rational$}{Q}-split flat} \label{MainSect}

\begin{defn}
Suppose $X/\Gamma$ is a locally symmetric space of noncompact type, and $\xi$ is a point on the visual boundary of~$X$.
We say the horospheres based at~$\xi$ are \emph{uniformly coarsely dense} in $X/\Gamma$ if there exists $C > 0$, such that, for every horosphere~$\mathcal{H}_t$ based at~$\xi$, every point of $X/\Gamma$ is at distance $< C$ from some point in $\pi(\mathcal{H}_t)$, where $\pi \colon X \to X/\Gamma$ is the natural covering map.
\end{defn}

\begin{rem}
Suppose $\Gamma_1 \subset \Gamma_2$. It is obvious that if the horospheres based at~$\xi$ are uniformly coarsely dense in $X/\Gamma_1$, then they are uniformly coarsely dense in $X/\Gamma_2$. \Cref{HoroLimIff} implies that the converse is true if $X/\Gamma_1$ has finite volume.
\end{rem}

\begin{thm} \label{HoroToInfty}
Let
\noprelistbreak
	\begin{itemize}
	\item $G = \GG(\real)^\circ$, where $\GG$ is a connected, semisimple $\rational$-group,
	\item $K \backslash G$ be the corresponding symmetric space of noncompact type with the Killing-form metric,
	\item $\Gamma$ be a subgroup of finite index in $G_{\integer}$,
	and
	\item $\xi$ be a point on the visual boundary of~$K \backslash G$.
	\end{itemize}
If the horospheres based at~$\xi$ are not uniformly coarsely dense in $K \backslash G/ \Gamma$, then there is a parabolic $\rational$-subgroup\/~$\PP$ of\/~$\GG$, such that 
\noprelistbreak
	\begin{enumerate}
	\item \label{HoroToInfty-FixXi}
	$\PP(\real)$ fixes~$\xi$,
	and
	\item \label{HoroToInfty-FixHoro}
	$\PP(\integer)$ fixes some\/ \textup(or, equivalently, every\/\textup) horosphere based at~$\xi$.
	\end{enumerate}
\end{thm}

\begin{proof}
Fix any $x  \in K \backslash G$.
Choose 
\noprelistbreak
	\begin{itemize}
	\item a maximal (connected) $\real$-split torus~$A$ of~$G$,
	and
	\item  a one-parameter subgroup $\{a^t\}$ of~$A$,
	\end{itemize}
such that
\noprelistbreak
	\begin{itemize}
	\item $xA$ is a (maximal) flat in $K \backslash G$,
	and
	\item $\xi$ is the endpoint of the geodesic ray $\{ x a^t\}_{t=0}^\infty$.
	\end{itemize}
Let	
\noprelistbreak
	\begin{itemize}
	\item $A^+$ be a Weyl chamber of~$A$ that contains $\{ a^t\}_{t=0}^\infty$,
	and
	\item $N = \left\{\, u \in G \mathrel{\Big|}
	\begin{matrix} \text{for all $a$ in the interior of~$A^+$,} \\
	\text{we have $a^k u a^{-k} \to e$ as $k \to +\infty$} \end{matrix}
	 \,\right\}$.
	\end{itemize}
Note that $G = KAN$ is an Iwasawa decomposition of~$G$.

Let $P = MTU$ and~$M^*$ be as in \cref{ClosureIsPorbit}. Denote by $A_\perp$ the orthogonal complement of  $\{a^t\}$ in~$A$ (with respect to the Killing form), so $A_\perp$ is a (codimension-one) connected subgroup of~$A$. Since $N \subseteq P$ (and $P$ is parabolic), we have $A \subset P$ (see \cref{ANinP}). Therefore, since all maximal $\real$-split tori of~$P$ are conjugate \cite[Thm.~20.9(ii), p.~228]{Borel-LinAlgGrps}, and $M^*T$ contains a maximal $\real$-split torus, there is no harm in assuming $A\subseteq M^*T$, by replacing $M^*T$ with a conjugate.

\fullCref{HoroSubgrp}{subgrp} tells us that the horosphere based at~$\xi$ through the point $x a^t$ is 
	$$\mathcal{H}_t = xa^t  A_\perp \, N .$$
(Note that $N$ preserves the horosphere and thus also the point $\xi$, so the proof of \pref{HoroToInfty-FixXi} will be complete when we show that the Levi subgroup $MT$ also preserves~$\xi$.)
We have 
	\begin{align*}
	\closure{ a^t  A_\perp \, N \Gamma }
	 \supseteq a^t A_\perp \cdot \closure{ N \Gamma}
	 = a^t A_\perp \cdot M^* \, U \, \Gamma
	. \end{align*}
By assumption, there is some~$t$, such that $\pi(\mathcal{H}_t)$ is not dense.
Assuming, as we may, that $K$ is the stabilizer of~$x$, this implies $K \cdot a^t A_\perp \cdot M^*U \neq G$.
Since $M^* T U \supseteq AN$ and $KAN = G$, we conclude that $T \not\subseteq A_\perp \, M^*$. (Note that this implies $P \neq G$.)

Let $A_M = A \cap M = A \cap M^*$, so $A = A_M T$. Then, since $T \not\subseteq A_\perp \, M^*$, we must have $A_\perp \, A_M \neq A$. Since $A_\perp$ has codimension one in~$A$, this implies $A_M \subseteq A_\perp$, which means $A_M \perp \{a^t\}$. On the other hand, \cref{SplitPerpNoChar} tells us $M \perp T$, which implies that $T$ is the orthogonal complement of~$A_M$ in~$A$. Therefore $\{a^t\} \subseteq T$, so $C_G(T) \subseteq C_G \bigl(\{a^t\}\bigr)$. Hence
	\begin{align*}
	P = MTU = C_G(T) \, U \subseteq C_G \bigl(\{a^t\}\bigr) \, N.
	\end{align*}
Since $C_G \bigl(\{a^t\}\bigr)$ and~$N$ each preserve the point $\xi$ at infinity, we conclude that the parabolic $\rational$-subgroup $P$ preserves the point $\xi$ at infinity. 
This completes the proof of \pref{HoroToInfty-FixXi}. 

Now, we turn to the proof of \pref{HoroToInfty-FixHoro}. Fixing a basepoint in~$K \backslash G$ yields a natural parametrization $\mathcal{H}_t$ of the horospheres based at~$\xi$. If $g$ is any isometry of $K \backslash G$ that fixes~$\xi$, then there is some $\ell = \ell(g)$, such that $\mathcal{H}_t g = \mathcal{H}_{t + \ell}$ for all~$t$. Thus, an isometry that fixes one of these horospheres must fix all of them.

Suppose there is some element $\gamma$ of~$\PP(\integer)$ with $\ell(\gamma) \neq 0$. (This will lead to a contradiction.) By replacing~$\gamma$ with a power of itself, we may assume $\gamma \in \Gamma$ (since $\ell(\gamma^n) = n \cdot \ell(\gamma) \neq 0$ for all $n \in \integer^+$). Then, for any $t \in \real$, we have
	$$ \mathcal{H}_t \cdot \Gamma \supset \mathcal{H}_t \cdot \langle \gamma \rangle = \bigcup_{n \in \integer} \mathcal{H}_{t + n \, \ell(\gamma)} .$$
Since every point in $K \backslash G$ is on some horosphere, this implies that every point is at distance less than $\ell(\gamma)$ from $\mathcal{H}_t \cdot \Gamma$. Therefore,  the horospheres based at~$\xi$ are uniformly coarsely dense in $K \backslash G / \Gamma$ (since $\ell(\gamma)$ is a constant, independent of~$t$). This is a contradiction.
\end{proof}

\begin{prop} \label{OnQsplit}
Assume the notation of \cref{HoroToInfty}.
If there is a parabolic\/ $\rational$-subgroup\/~$\PP$ of\/~$\GG$, such that	
\noprelistbreak
	\begin{itemize}
	\item $\PP(\real)$ fixes ~$\xi$, 
	and
	\item $\PP(\integer)$ fixes every horosphere based at~$\xi$, 
	\end{itemize}
then $\xi$ is on the boundary of a\/ $\rational$-split flat.
\end{prop}

\begin{proof} 
Let $P = \PP(\real)$. There exists a $\rational$-torus~$T$ of~$P$, such that $T$ contains a maximal $\real$-split torus~$A$ \cite[Cor.~3 of \S7.1, p.~405]{PlatonovRapinchukBook}. Choose $x \in K \backslash G$, such that $x A$ is a (maximal) flat. Since $A \subset P$ fixes~$\xi$, \Cref{AFixesXi} provides a geodesic $\gamma = \{\gamma_t\}$ in~$xA$, such that $\lim_{t \to \infty} \gamma_t = \xi$ (and $\gamma_0 = x$). 

Write $T = S  E$, where $S$ is $\rational$-split and $E$ is $\rational$-anisotropic. Then $E_{\integer}$ is a cocompact lattice in~$E$ \cite[Thm.~4.11, p.~208]{PlatonovRapinchukBook} and, by assumption, $E_{\integer}$ fixes the horosphere through~$x$. This implies that all of~$E$ fixes this horosphere, so the flat $xE$ is contained in the horosphere, and is therefore perpendicular to the geodesic~$\gamma$. Since \cref{SplitPerpNoChar} tells us that the orthogonal complement of~$xE$ is $xS$, we conclude that $\gamma \subseteq xS$. So $\xi$ is on the boundary of the $\rational$-split flat $xS$.
\end{proof}

\begin{cor} \label{HoroLimIff}
Let $X/\Gamma$ be an arithmetic locally symmetric space of noncompact type with the Killing-form metric, and let $\xi$ be a point on the visual boundary of~$X$. Then the following are equivalent:
\noprelistbreak
	\begin{enumerate}
	\item \label{HoroLimIff-horo}
	$\xi$ is a horospherical limit point for\/~$\Gamma$.
	\item \label{HoroLimIff-notflat}
	$\xi$ is not on the boundary of any\/ $\rational$-split flat.
	\item \label{HoroLimIff-parab}
	There does not exist a parabolic $\rational$-subgroup\/~$\PP$ of\/~$\GG$, such that\/ $\PP(\real)$ fixes~$\xi$, and\/ $\PP(\integer)$ fixes some\/ \textup(or, equivalently, every\/\textup) horosphere based at~$\xi$.
	\item \label{HoroLimIff-SpheresDense}
	The horospheres based at~$\xi$ are uniformly coarsely dense in $X/\Gamma$.
	\item \label{HoroLimIff-BallsDense}
	The horoballs based at~$\xi$ are uniformly coarsely dense in $X/\Gamma$.
	\item \label{HoroLimIff-onto}
	$\pi(\mathcal{B}) = X/\Gamma$ for every horoball~$\mathcal{B}$ based at~$\xi$, where $\pi \colon X \to X/\Gamma$ is the natural covering map.
	\end{enumerate}
\end{cor}

\begin{proof}
(\ref{HoroLimIff-horo}~$\Rightarrow$~\ref{HoroLimIff-notflat}) is the contrapositive of \cref{OnFlatNotHoro}. 
(\ref{HoroLimIff-notflat}~$\Rightarrow$~\ref{HoroLimIff-parab}) is the contrapositive of \cref{OnQsplit}.
(\ref{HoroLimIff-parab}~$\Rightarrow$~\ref{HoroLimIff-SpheresDense}) is the contrapositive of \cref{HoroToInfty}.
(\ref{HoroLimIff-SpheresDense}~$\Rightarrow$~\ref{HoroLimIff-BallsDense}) is obvious,  because horoballs are bigger than horospheres.
(\ref{HoroLimIff-BallsDense}~$\Rightarrow$~\ref{HoroLimIff-horo}) is \cref{CGammaHoroLimDefn}($\Leftarrow$).
(\ref{HoroLimIff-horo}~$\Leftrightarrow$~\ref{HoroLimIff-onto}) is a restatement of \cref{HoroLimDefn}.
\end{proof}

\begin{rem} 
Suppose $\GG$ is $\rational$-split (or, more generally, suppose $\rank{\rational} \GG = \rank{\real} \GG$).
Under this assumption, it is easy to show that if $\xi$ is not on the boundary of a $\rational$-split flat, then every horosphere based at~$\xi$ is dense in $K \backslash G / \Gamma$, not just coarsely dense. To see this, we prove the contrapositive. Note that the proof of \fullcref{HoroToInfty}{FixXi} only assumes there is a horosphere that is not dense. Now, if we let $S$ be any maximal $\rational$-split torus of~$P$, then $S$ is also a maximal $\real$-split torus (by our assumption that $\rank{\rational} \GG = \rank{\real} \GG$), so \cref{AFixesXi} tells us that $\xi$ is on the boundary of the corresponding ($\rational$-split) maximal flat $xS$ (since $S \subset P$ fixes~$\xi$).
\end{rem}

\section{Non-arithmetic locally symmetric spaces} \label{NonarithSect}

To state a generalization of \cref{HoroLimIffgRat} that does not require $X/\Gamma$ to be arithmetic, we need an appropriate generalization of the notion of a $\rational$-split flat.

\begin{defn}
Let $X / \Gamma = K \backslash G / \Gamma$ be a finite-volume locally symmetric space of noncompact type. 
\noprelistbreak
	\begin{itemize}
	\item A parabolic subgroup~$P$ of~$G$ is \emph{$\Gamma$-rational} if $\Gamma$ contains a lattice subgroup of~$\unip P$.
	
	\item Assume $\rank{\real} G = 1$. A torus~$S$ in~$G$ is \emph{$\Gamma$-split} if $S$ is $\real$-split and $S$ is contained in the intersection of two different $\Gamma$-rational, proper, parabolic subgroups of~$G$.
	
	\item From the Margulis Arithmeticity Theorem \cite[Thm.~1, p.~2]{MargulisBook}, we know that, after passing to a finite cover of~$X/\Gamma$ (in other words, after passing to a finite-index subgroup of~$\Gamma$), we can write 
		$$X/\Gamma =  (X_a/\Gamma_a) \times (X_c / \Gamma_c) \times (X_1/\Gamma_1) \times \cdots \times (X_n / \Gamma_n) ,$$
	where $X_a/\Gamma_a$ is arithmetic, $X_c / \Gamma_c$ is compact with all factors of real rank one, and each $X_k/\Gamma_k$ is noncompact with real rank one. 
	A torus in~$G$ is \emph{$\Gamma$-split} if it is contained in some torus of the form $S_a \times \{e\} \times S_1 \times \cdots \times S_n$, where $S_a$ is $\rational$-split, and each $S_k$ is $\Gamma_k$-split.
	
	\item A flat $x T$ in~$X$ is \emph{$\Gamma$-split} if the torus $T$ is $\Gamma$-split.
	\end{itemize}
\end{defn}

\begin{rem}
Suppose $G$ is defined over~$\rational$. It can be shown that:
\noprelistbreak
	\begin{enumerate}
	\item a parabolic subgroup of~$G$ is $G_{\integer}$-rational if and only if it is defined over~$\rational$,
	and
	\item a torus in~$G$ is $G_{\integer}$-split if and if it is $\rational$-split.
	\end{enumerate}
\end{rem}

A slight modification of the above arguments establishes the following generalization of \cref{HoroLimIff}. 

\begin{prop} \label{HoroLimIffNonArith}
Let $X/\Gamma$ be a finite-volume locally symmetric space of noncompact type with the Killing-form metric, and let $\xi$ be a point on the visual boundary of~$X$. Then the following are equivalent:
\noprelistbreak
	\begin{enumerate}
	\item $\xi$ is a horospherical limit point for\/~$\Gamma$.
	\item $\xi$ is not on the boundary of any\/ $\Gamma$-split flat.
	\item There does not exist a $\Gamma$-rational parabolic subgroup~$P$ of~$G$, such that $P$ fixes~$\xi$, and\/ $P \cap \Gamma$ fixes some\/ \textup(or, equivalently, every\/\textup) horosphere based at~$\xi$.
	\item The horospheres based at~$\xi$ are uniformly coarsely dense in $X/\Gamma$.
	\item The horoballs based at~$\xi$ are uniformly coarsely dense in $X/\Gamma$.
	\item $\pi(\mathcal{B}) = X/\Gamma$ for every horoball~$\mathcal{B}$ based at~$\xi$, where $\pi \colon X \to X/\Gamma$ is the natural covering map.
	\end{enumerate}
\end{prop}

\begin{rem} \label{NotKilling}
The above results apply only to the Killing-form metric on $K \backslash G$, but it is well known that any other symmetric metric~$g$ differs only by a scalar multiple on each irreducible factor of $K \backslash G$ \cite[p.~378]{Helgason-DiffGeom}.
If the endpoint of a particular geodesic ray $\{x a^t\}_{t=0}^\infty$ is a horospherical limit point in the Killing-form metric, and we let
	$$ b^t = (a_1^{t/\lambda_1}, \ldots, a_n^{t/\lambda_n}) , $$
where $\lambda_i$ is the scaling factor of~$g$ on the irreducible factor $K_i \backslash G_i$, then the endpoint of $\{x b^t\}_{t=0}^\infty$ is a horospherical limit point with respect to the metric~$g$. In fact, it is easy to see that the two different geodesic rays (in the two different metrics) have exactly the same horospheres in $K \backslash G$.

This means that the above proofs apply in general if we replace the phrase ``$\rational$-split'' with ``$\rational$-good,'' where a torus~$S$ is \emph{$\rational$-good} if $S$ is contained in a maximal $\rational$-torus~$T$ of~$G$, such that $T$ contains a maximal $\rational$-split torus of~$G$, and $S$~is orthogonal to the maximal $\rational$-anisotropic torus of~$T$ (cf.\ \cref{SplitPerpNoChar}). 
\end{rem}

\end{document}